\documentclass[a4paper]{amsart}

\usepackage{latexsym, amssymb,amsmath, amsthm,amsfonts,tikz}

\newcommand{\kk}{\Bbbk}

\def\Hom{\operatorname{Hom}}

\def\Z{\mathbb{Z}}
\def\N{\mathbb{N}}

\def\im{\operatorname{im}}
\def\coker{\operatorname{coker}}
\def\Ext{\operatorname{Ext}}

\def\Tr{\operatorname{Tr}}
\def\Id{\operatorname{id}}

\newtheorem{Lemma}{Lemma}
\newtheorem{Theorem}[Lemma]{Theorem}

\newtheorem{Corollary}[Lemma]{Corollary}

\newtheorem{prop}[Lemma]{Proposition}

\newtheorem*{Corollary of Conjecture}{Corollary of Conjecture}

\theoremstyle{definition}

\theoremstyle{remark}
  \newtheorem{rem}[Lemma]{Remark}

\newtheoremstyle{Acknowledgments}
  {}
    {}
     {}
     {}
    {\bfseries}
    {}
     {.5em}
     {\thmname{#1}\thmnumber{ }\thmnote{ (#3)}}
\theoremstyle{Acknowledgments}

\title{The relative Heller operator and relative cohomology for the Klein 4-group.}

\author{Jonathan Elmer}
\address{Middlesex University\\
The Burroughs, Hendon, London\\
NW4 4BT UK}
\email{j.elmer@mdx.ac.uk}

\date{\today}

\begin{document}

\maketitle

\begin{abstract} Let $G$ be the Klein Four-group and let $\kk$ be an arbitrary field of characteristic 2. A classification of indecomposable $\kk G$-modules is known. We calculate the relative cohomology groups $H^i_{\chi}(G,N)$ for every indecomposable $\kk G$-module $N$, where $\chi$ is the set of proper subgroups in $G$. This extends work of Pamuk and Yalcin to cohomology with non-trivial coefficients. We also show that all cup products in strictly positive degree in $H_{\chi}^*(G,\kk)$ are trivial. 
\end{abstract}

\section{Introduction}
Let $G$ be a finite group and $\kk$ a field of characteristic $p>0$. If $p \not | |G|$, then every representation of $G$ over $\kk$ is projective. Thus, by decomposing the regular module $\kk G$ we can obtain all isomorphism classes of $\kk G$-modules immediately.

From now on assume $p | |G|$. Then the above is no longer true. However, it is well-known that, given a $\kk G$-module $M$, we can find a projective module $P_0$ and a surjective $\kk G$-morphism
\[\pi_0: P_0 \rightarrow M.\]
If we choose $P_0$ and $\pi_0$ so that $P_0$ has smallest possible dimension, then this pair is unique, and known as the projective cover of $M$. The kernel of $\pi_0$ is denoted $\Omega(M).$ The is known as the Heller shift of $M$. $\Omega(-)$ can be viewed as an operation on the set of $\kk G$-modules which takes indecomposable modules to indecomposable modules.

This construction can be iterated. For each $i>0$, let $\pi_i: P_i \rightarrow \Omega^i(M)$ be the projective cover of $\Omega^i(M)$. By composing these maps with the inclusions $\Omega^i(M) \rightarrow P_{i-1}$, we obtain an exact sequence 

\begin{equation}\label{projres} \ldots P_i \rightarrow P_{i-1} \rightarrow \ldots \rightarrow P_0 \rightarrow M \rightarrow 0.\end{equation}  

This is an example of a projective resolution for $M$. If $N$ is any $\kk G$-module, then the above induces a complex

\[0 \rightarrow \Hom_{\kk G}(P_0,M) \rightarrow \ldots \rightarrow \Hom_{\kk G}(P_i, M) \rightarrow \ldots\]

which is not exact in general. The homology groups of this complex are by definition the groups $\Ext_{\kk G}^i(M,N)$. A special case is
\[H^i(G,N):= \Ext^i_{\kk G} (\kk, N).\]
We call this the \emph{cohomology of $G$ with coefficients in $N$}.

There is a long and fruitful history of study of the cohomology groups $H^i(G,N)$ in modular representation theory. Further, one may define a pairing
\[\smile: H^i(G,\kk) \otimes H^j(G, \kk) \rightarrow H^{i+j}(G,\kk)\] which gives $H^*(G,\kk)$ the structure of a graded-commutative graded ring. A celebrated theorem of Evens (see \cite[Theorem~4.2.1]{Benson2}) states that, for any $G$, the ring $H^*(G,\kk)$ is finitely generated. 

Now let $\chi$ be a set of proper subgroups of $G$. A $\kk G$-module $M$ is said to be \emph{projective relative to $\chi$} if $M$ is  a direct summand of $\oplus_{X \in \chi} M \downarrow_X \uparrow^G$. Other equivalent definitions will be given in section 2. It is less well-known, but still true, that every $\kk G$-module has a unique relative projective cover with respect to $\chi$. This is defined to be a $\kk G$-module $Q_0$ of smallest dimension such that

\begin{enumerate} 
\item $Q_0$ is projective relative to $\chi$;
\item There is a surjective $\kk G$-morphism $\pi_0: Q_0 \rightarrow M$ which splits on restriction to each $X \in \chi$.
\end{enumerate}

The kernel of $Q_0$ is denoted $\Omega_{\chi}(M)$ and called the relative Heller shift of $M$ with respect to $\chi$. We can mimic the construction of (\ref{projres}) to obtain a relative projective resolution of $M$, that is, an exact sequence

\begin{equation}\label{relprojres} \ldots Q_i \rightarrow Q_{i-1} \rightarrow \ldots \rightarrow Q_0 \rightarrow M \rightarrow 0.\end{equation}  

of $\kk G$ modules which are projective relative to $\chi$ and in which the connecting homomorphisms split over each $X \in \chi$. Given any $\kk G$-module $M$, the above induces a complex

\[0 \rightarrow  \Hom_{\kk G}(Q_0,M) \rightarrow \ldots \rightarrow \Hom_{\kk G}(Q_i, M) \rightarrow \ldots\]

which is in general no longer exact. The homology groups of this complex are by definition the relative Ext-groups $\Ext_{\kk G, \chi}^i(M,N)$. The relative cohomology of $G$ with respect to $\chi$ with coefficients in $N$ is the special case
\[H^i_{\chi}(G, N): = \Ext_{\kk G, \chi}^i(\kk, N).\]

Further, one may define a pairing \[\smile: H_{\chi}^i(G,\kk) \otimes H_{\chi}^j(G, \kk) \rightarrow H_{\chi}^{i+j}(G,\kk)\] which gives $H^*(G,\kk)$ the structure of a graded-commutative graded ring.

Computations of $H_{\chi}^i(G,N)$ are rare in the literature. It is notable that the ring $H^*_{\chi}(G, \kk)$ is not finitely generated in general. This was first discovered by Blowers \cite{Blowers}, who showed that if $G_1$ and $G_2$ are finite groups of order divisible by $p$, and $\chi_1, \chi_2$ are sets of subgroups of $G_1, G_2$ respectively with order divisible by $p$, then all products of elements of positive degree in $H^*_{\chi}(G, \kk)$ are zero, where $G = G_1 \times G_2$ and $\chi = \{G_1 \times X: X \in \chi_2\} \cup \{X \times G_2: X \in \chi_1\}$. See also \cite{Brown}.

For the rest of this section, let $G = \langle \sigma, \tau \rangle$ denote the Klein four-group, and let $\kk$ be a field of characteristic 2. We set $\chi = \{H_1,H_2,H_3\}$, the set of all proper nontrivial subgroups of $G$, where $H_1 = \langle \sigma \rangle, H_2 = \langle \tau \rangle, H_3 = \langle \sigma \tau \rangle$. 

The cohomology groups $H_{\chi}^i(G,\kk)$ were computed, by indirect means, by Pamuk and Yalcin \cite{PamukYalcin}. In the present article we recover their result, and also compute $H^i_{\chi}(G,N)$ for any $\kk G$-module $N$. Our methods are more direct; we compute an explicit relative projective resolution for each $N$. Of course we are helped enormously by the fact that the representations of $G$ are completely classified. Our first main result is:

\begin{Theorem}\label{heller}
Let $M$ be an indecomposable $\kk G$-module, which is not projective relative to $\chi$. Then we have
\[\Omega_{\chi}(M) \cong \Omega^{-2}(M)\] if $M$ has odd dimension, and
\[\Omega_{\chi}(M) \cong M\] otherwise.
\end{Theorem}
 
The ring structure of $H^*_{\chi}(G,\kk)$ was not considered in \cite{PamukYalcin}. Note, however, that if $\chi'$ is a subset of $\chi$ with size 2, then all products in $H_{\chi'}^*(G,\kk)$ are zero, by a special case of Blowers' result. It is perhaps not surprising, therefore, that we have

\begin{Theorem}\label{cupprod}
Let $\alpha_1, \alpha_2 \in H_{\chi}^*(G, \kk)$, where both have strictly positive degree. Then $\alpha_1 \smile \alpha_2= 0$.
\end{Theorem}

This paper is organised as follows. In section 2 we define relative projectivity and derive the results we will need to do the computations in later sections. This section follows \cite[Section~2]{ElmerSymPowers} fairly closely. As most proofs can be constructed by adapting familiar results projectivity to the relative case, they are omitted. In section 3 we describe the classification of modules for the Klein-four group and prove Theorem \ref{heller}. We also compute $H_{\chi}^i(G,N)$ for every $\kk G$-module $N$ and prove Theorem \ref{cupprod}. 
 
\subsection{Notation} All groups under consideration are finite groups, and for any group $G$, by a $\kk G$-module we mean a finitely-generated $\kk$-vector space with compatible $G$ action. The one-dimensional trivial $\kk G$-module will be denoted by $\kk_G$ or simply $\kk$ when the group acting is obvious, and for $n\in \N$ and $M$ a $\kk G$-module we write $nM$ for the direct sum of $n$ copies of $M$.

\section{Relative projectivity}

In this section, let $p>0$ be a prime and let $G$ be a finite group of order divisible by $p$. Let $\kk$ be a field of characteristic $p$ and let $\chi$ be a set of subgroups of $G$.
Now let $M$ be a finitely generated $\kk G$-module. $M$ is said to be \emph{projective relative to $\chi$} if the following holds: let $\phi:M \rightarrow Y$ be a $\kk G$-homomorphism and $j: X \rightarrow Y$ a surjective $\kk G$-homomorphism which splits on restriction to any subgroup of $H \in \chi$. Then there exists a $\kk G$-homomorphism $\psi$ making the following diagram commute.

\begin{center}
\begin{tikzpicture}
\draw[->] (0,0) node[left] {$X$} -- (2,0) node[right] {$Y$}; 
\draw[->] (2.5,0) -- (4.5,0) node[right] {0};
\draw[->] (2.2,2) node[above] {$M$} -- (2.2,0.3);
\draw[dashed,->] (2,2) -- (0,0.3);
\draw (0.8,1.2) node[above]{$\psi$};
\draw (1,0) node[above] {$j$};
\draw (2.2,1) node[right] {$\phi$};
\end{tikzpicture}
\end{center}  

Dually, one says that $M$ is \emph{injective relative to $\chi$} if the following holds: given an injective $\kk G$-homomorphism $i: X \rightarrow Y$ which splits on restriction to each $H \in \chi$ and a $\kk G$-homomorphism $\phi:  X \rightarrow M$, there exists a $\kk G$-homomorphism $\psi$ making the following diagram commute.

\begin{center}
\begin{tikzpicture}
\draw[->] (0,0) node[left] {$X$} -- (2,0) node[right] {$Y$}; 
\draw[<-] (-0.5,0) -- (-2.5,0) node[left] {0};
\draw[->] (-0.2,-0.2) -- (-0.2,-2)node[below] {$M$};
\draw[dashed,->] (2,-0.2) -- (0.2,-2);
\draw (1.2,-1.2) node[below]{$\psi$};
\draw (1,0) node[above] {$i$};
\draw (-0.2,-1) node[left] {$\phi$};
\end{tikzpicture}
\end{center}  

These notions are equivalent to the usual definitions of projective and injective $\kk G$-modules when we take $\chi = \{1\}$. We will say a $\kk G$-homomorphism is $\chi$-split if it splits on restriction to each $H \in \chi$.  Since a $\kk G$-module is projective relative to $H$ if and only if it is also projective relative to the set of all subgroups of $H$, we usually assume $\chi$ is closed under taking subgroups.

We denote the set of $G$-fixed points in $M$ by $M^G$. For any $H \leq G$ there is a $\kk G$-map $M^H \rightarrow M^G$ defined as follows:
\[\Tr^G_H(x) = \sum_{\sigma \in S} \sigma x\] where $x \in M$ and $S$ is a left-transversal of $H$ in $G$. This is called the relative trace or transfer. It is clear that the map is independent of the choice of $S$. If $H=1$ we usually write this as $\Tr^G$ and call it simply the trace or transfer. For any set of subgroups $\chi$ of $G$ we define the subspace
\[M^{G, \chi}:= \sum_{H \in \chi} \Tr^H_G(M^H)\]
and quotient
\[M^G_{\chi}:= \frac{M^G}{M^{G,\chi}}.\]

Now let $N$ be another $\kk G$-module. We can define an action of $G$ on $\Hom_{\kk}(M,N)$: 
\[(g \cdot \phi)(x) = g \phi(g^{-1}x) \ \text{for} \ g \in G, x \in M.\]
Notice that with this action we have $\Hom_{\kk}(M,N)^G = \Hom_{\kk G}(M,N)$. Further, the transfer construction gives a map $$\Tr^G_H: \Hom_{\kk H}(M,N) \rightarrow \Hom_{\kk G}(M,N).$$

There a various ways to characterize relative projectivity:
\begin{prop}\label{omnibus}
Let $G$ be a finite group of order divisible by $p$, $\chi$ a set of subgroups of $G$ and $M$ a $\kk G$-module. Then the following are equivalent:
\begin{enumerate}
\item[(i)] $M$ is projective relative to $\chi$;
\item[(ii)] Every $\chi$-split epimorphism of $\kk G$-modules $\phi: N \rightarrow M$ splits;
\item[(iii)] $M$ is injective relative to $\chi$;
\item[(iv)] Every $\chi$-split monomorphism of $\kk G$-modules $\phi: M \rightarrow N$ splits;
\item[(v)] $M$ is a direct summand of $\oplus_{H \in \chi} M \downarrow_H \uparrow^G$;
\item[(vi)] $M$ is a direct summand of a direct sum of modules induced from subgroups in $\chi$
\item[(vii)] There exists a set of homomorphisms $\{\beta_{H}: H \in \chi\}$  such that $\beta_H \in \Hom_{\kk H}(M,M)$ and $\sum_{H \in \chi} \Tr_H^G(\beta_H) = \Id_M$.
\end{enumerate}
\end{prop}
 
The last of these is called \emph{Higman's criterion}.
\begin{proof} The proof when $\chi$ consists of a single subgroup of $G$ can be found in \cite[Proposition~3.6.4]{Benson1}. This can easily be generalised. 
\end{proof}

For homomorphisms $\alpha \in \Hom_{\kk G}(M,N)$ we have the following:
\begin{Lemma}\label{factorsthrough} Let $M$, $N$ be $\kk G$-modules, $\chi$ a collection of subgroups of $G$, and $\alpha \in \Hom_{\kk G}(M,N)$. Then the following are equivalent:

\begin{enumerate}
\item[(i)] $\alpha$ factors through $\oplus_{H \in \chi} M \downarrow_H \uparrow^G$.
\item[(ii)] $\alpha$ factors through some module which is projective relative to $\chi$.
\item[(iii)] There exist homomorphisms $\{\beta_H \in \Hom_{\kk H}(M,N): H \in \chi\}$ such that $\alpha = \sum_{H \in \chi} \Tr^G_H(\beta_H)$. 
\end{enumerate}
 \end{Lemma}

\begin{proof} This is easily deduced from \cite[Proposition 3.6.6]{Benson1}.
\end{proof}

The above tells us that $\Hom_{\kk}(M,N)^{G,{\chi}}$ consists of the $\kk G$-homomorphsims which factor through a module which is projective relative to $\chi$. We write
\[\underline{\Hom}^{\chi}_{\kk G}(M,N):= \Hom_{\kk}(M,N)^G_{\chi}.\] 


Let $M$ be a $\kk G$-module and let $X$ be a $\kk G$-module that is projective relative to $\chi$. It is easily shown, using Proposition \ref{omnibus}, that $M \otimes X$ is projective relative to $\chi$. For example, the module $M \otimes X$ where $X = \bigoplus_{H \in \chi} \kk_{H} \uparrow^G$ is projective relative to $\chi$. Moreover, the natural map $\sigma: M \otimes X \rightarrow M$ given by
\[\sigma(m \otimes x)=m \] is a $\chi$-split $\kk G$-epimorphism (to see the splitting, use the Mackey Theorem). It follows that for each $M$, there exists a $\kk G$-module $Q_0$ which is projective relative to $\chi$ and a $\chi$-split $\kk G$-epimorphism $\pi_0: Q_0 \rightarrow M$. 

Let $\pi_0: Q_0 \rightarrow M$ and $\pi'_0: \rightarrow Q_0' \rightarrow M$ be two such pairs. The proof of Schanuel's Lemma (see \cite[Lemma~1.5.3, Lemma~3.9.1]{Benson1}) extends more or less verbatim to the relative case; if $K_0 = \ker(\pi)$ and $K_0' = \ker(\pi'_0)$ then $K_0 \oplus Q'_0 \cong K_0' \oplus Q_0$. 

If we choose among all such pairs, one in which the dimension of $Q_0$ is minimal, the kernel $K_0$ is defined uniquely. This pair $(Q_0,\pi_0)$ is called the relative projective cover of $M$. For this choice we set $\Omega_{\chi}(M) = K_0$. We can interate this construction, setting $\Omega_{\chi}^i(M) = \Omega_{\chi}(\Omega^{i-1}_{\chi}(M))$. Minimality implies that if $K_0'$ is the kernel of any other $\chi$-split $\kk G$-epimorphism $Q_0' \rightarrow M$, then $K_0 \cong \Omega_{\chi}(M) \oplus \textrm(rel. proj)$, where (rel. proj) is some module which is projective relative to $\chi$.

Dually, we always have that $M$ is a submodule of $M \otimes X$ with $X =  \bigoplus_{H \in \chi} \kk_{H} \uparrow^G$, and the inclusion $\rho: M \rightarrow M \otimes X$ splits on restriction to each $H \in \chi$. It follows that for each $M$, there exists a $\kk G$-module $J_0$ and a $\chi$-split $\kk G$-monomorphism $\rho_0: M \rightarrow J_0$. 

Let $\rho_0: M \rightarrow J_0$ and $\rho'_0: M \rightarrow J_0'$ be two such pairs. Again. by the relative version of Schanuel's Lemma, if $C_0 = \coker(\pi)$ and $C_0' = \coker(\pi'_0)$ then $C_0 \oplus J'_0 \cong C_0' \oplus J_0$. 

If we choose among all such pairs, one in which the dimension of $J_0$ is minimal, the cokernel $C_0$ is defined uniquely. The pair $(J_0,\rho_0)$ is called a relative injective hull of $M$ with respect to $\chi$. For this choice we set $\Omega^{-1}_{\chi}(M) = K_0$. We can interate this construction, setting $\Omega_{\chi}^{-i}(M) = \Omega^{-1}_{\chi}(\Omega^{-(i-1)}_{\chi}(M))$. Minimality implies that if $K_0'$ is the kernel of any other $\chi$-split $\kk G$-monomorphism $M \rightarrow J_0$, then $K'_0 \cong \Omega^{-1}_{\chi}(M) \oplus \textrm(rel. proj)$, where (rel. proj) is some module which is projective relative to $\chi$.

The following gives some properties of the operators $\Omega^i_{\chi}$.  

\begin{prop}\label{omegaomnibus} Let $M_1, M_2$ be $\kk G$-modules without summands which are  projective relative to $\chi$, and $i, j$ nonzero integers. Then:

\begin{enumerate}
\item[(i)] $\Omega^i_{\chi}(M_1 \oplus M_2) \cong \Omega^i_{\chi}(M_1) \oplus \Omega^i_{\chi}(M_2)$;
\item[(ii)] $\Omega_{\chi}^i(M)^* \cong \Omega^{-i}_{\chi}(M^*)$;
\item[(iii)] $M \cong \Omega_{\chi}(\Omega_{\chi}^{-1}(M)) \oplus$ (rel. proj) $\cong \Omega^{-1}_{\chi}(\Omega_{\chi}(M)) \oplus $ (rel. proj.).
\end{enumerate}
\end{prop}

\begin{proof} (i) is obvious. (ii,iii) are easily deduced from the relative version of Schanuel's Lemma.
\end{proof}

(i) above shows that $\Omega_{\chi}^i$ is a well-defined operator on the set of indecomposable $\kk G$-modules which are not relatively projective to $\chi$. Note that (iii) does not say that $\Omega_{\chi} \circ \Omega^{-1}_{\chi}$ is the identity in general. If we define $\Omega^0_{\chi}(M)$ to be the direct sum of all summands of $M$ which are not projective relative to $\chi$, then we have $\Omega^{i+j} = \Omega_{\chi}^i \circ \Omega_{\chi}^j$ for all $i$ and $j$. 

The following result is sometimes useful.
\begin{Lemma}\label{relprojinduced} Let $M$ be a $\kk G$-module which is projective relative to a set $\chi$ of subgroups of $G$. Then $M^G = \sum_{H \in \chi} \Tr^G_H(M^H)$. 
\end{Lemma}

\begin{proof} See \cite[Lemma~2.9]{ElmerSymPowers}
\end{proof}

As a consequence of the above, if $M = N \oplus$ (rel. proj.), we get that $M^G_{\chi} = N^G_{\chi}$.
The operators $\Omega^i_{\chi}$ extend in a natural way to homomorphisms between modules. Let $f \in \Hom_{\kk G}(M,N)$. Let $(Q,\pi),(Q',\pi')$ be the relative projective covers of $M,N$. Then the relative projectivity of $Q$ ensures the existence of a homomorphism $\bar{f} \in \Hom_{\kk G}(Q,Q')$ making the following diagram commute\\

\begin{tikzpicture}
\draw[->] (0.5,0) node[left] {$\Omega_{\chi}(M)$} -- (3,0) node [right] {$Q$};
\draw[->] (3.5,0) node[right]{} -- (6,0) node[right] {$M$};
\node[above] at (4.75,0){$\pi$};
\draw[->] (6.6,0) -- (9,0) node[right]{$0$};
\draw[->] (0.5,-3) node[left] {$\Omega_{\chi}(N)$} -- (3,-3) node [right] {$Q'$};
\draw[->] (3.5,-3) node[right]{} -- (6,-3) node[right] {$N$};
\node[above] at (4.75,-3){$\pi'$};
\draw[->] (6.6,-3) -- (9,-3) node[right]{$0$};
\draw[->] (6.3,-0.3) -- (6.3,-2.7);
\node[left] at (6,-1.5){$f$};
\draw[->,dashed] (3.3,-0.3) -- (3.3,-2.7);
\node[left] at (3,-1.5){$\bar{f}$};
\draw[->,dashed] (0.3,-0.3) -- (0.3,-2.7);
\node[left] at (0,-1.5){$\Omega_{\chi}(f)$};
\end{tikzpicture}

and an easy diagram chase shows that the image of $\Omega_{\chi}(f): \bar{f}|_{\ker(\pi)}$ is contained in $\ker(\pi')$. In this way, $f$ induces a homomorphism $\Omega_{\chi}(f) \in \Hom_{\kk G}(\Omega_{\chi}(M),\Omega_{\chi}(N))$. 

In a similar fashion, let $(J,\rho),(J',\rho')$ be the relative injective hulls of $M,N$ respectively. Then relative injectivity of $J'$ ensures the existence of a homomorphism $\tilde{f} \in \Hom(J,J')$ making the following diagram commute,

\begin{tikzpicture}
\draw[<-] (0.5,0) node[left] {$\Omega_{\chi}(M)$} -- (3,0) node [right] {$Q$};
\draw[<-] (3.5,0) node[right]{} -- (6,0) node[right] {$M$};
\node[above] at (4.75,0){$\rho$};
\draw[<-] (6.6,0) -- (9,0) node[right]{$0$};
\draw[<-] (0.5,-3) node[left] {$\Omega^{-1}_{\chi}(N)$} -- (3,-3) node [right] {$Q'$};
\draw[<-] (3.5,-3) node[right]{} -- (6,-3) node[right] {$N$};
\node[above] at (4.75,-3){$\rho'$};
\draw[<-] (6.6,-3) -- (9,-3) node[right]{$0$};
\draw[->] (6.3,-0.3) -- (6.3,-2.7);
\node[left] at (6,-1.5){$f$};
\draw[->,dashed] (3.3,-0.3) -- (3.3,-2.7);
\node[left] at (3,-1.5){$\tilde{f}$};
\draw[->,dashed] (0.3,-0.3) -- (0.3,-2.7);
\node[left] at (0,-1.5){$\Omega^{-1}_{\chi}(f)$};
\end{tikzpicture}

and a diagram chase shows that $\tilde{f}$ induces a well-defined homomorphism $\Omega_{\chi}^{-1}(f) \in \Hom(\Omega^{-1}_{\chi}(M), \Omega_{\chi}^{-1}(N))$. For a given homomorphism $f$, $\Omega_{\chi}^{-1}(\Omega_{\chi}(f))$ and $\Omega_{\chi}(\Omega_{\chi}^{-1}(f))$ are not equal to $f$ in general, but $f-\Omega_{\chi}^{-1}(\Omega_{\chi}(f))$ and $f-\Omega_{\chi}(\Omega_{\chi}^{-1}(f))$ factor through a module which is projective relative to $\chi$. By the discussion following Lemma \ref{factorsthrough}, we have

\begin{prop}\label{suspendhom} For all $i \in \Z$, $\Omega_{\chi}^i(-)$ induces an isomorphism 
\[\underline{\Hom}_{\kk G}^{\chi}(M,N) \cong \underline{\Hom}_{\kk G}^{\chi}(\Omega^i_{\chi}(M),\Omega^i_{\chi}(N)).\]
\end{prop}

As explained in the introduction, the idea of a relatively projective cover can be extended to a relatively projective resolution; that is, an exact complex

\begin{equation} \ldots \rightarrow Q_i \rightarrow Q_{i-1} \rightarrow \ldots \rightarrow Q_0 \rightarrow M \rightarrow 0\end{equation}  
of relatively projective modules in which the connecting homomorphisms split over $\chi$. If

\begin{equation} \ldots \rightarrow Q'_i \rightarrow Q'_{i-1} \rightarrow \ldots \rightarrow Q'_0 \rightarrow M \rightarrow 0\end{equation}  
is another relatively projective resolution, then it turns out that any two chain maps between them are chain homotopic (see \cite[Theorem~3.9.3]{Benson1} for the version with $\chi$ consisting of one subgroup - the proof of the more general version is the same). Consequently, for any $\kk G$-module $N$, the homology groups of the induced complex

\[0 \rightarrow \Hom_{\kk G}(Q_0,M) \rightarrow \ldots \rightarrow \Hom_{\kk G}(Q_i, M) \rightarrow \ldots\]

are independent of the choice of resolution. The homology groups of this complex are by definition the relative Ext-groups $\Ext_{\kk G, \chi}^i(M,N)$. The relative cohomology of $G$ with respect to $\chi$ with coefficients in $N$ is the special case
\[H^i_{\chi}(G, N): = \Ext_{\kk G, \chi}^i(\kk, N).\]

We will use a minimal relative projective resolution of the trivial module to compute relative cohomology; that is, a relatively projective resolution
\begin{equation} \ldots \rightarrow Q_i \stackrel{\partial_{i-1}}{\rightarrow} Q_{i-1} \rightarrow \ldots \stackrel{\partial_0}{\rightarrow} Q_0 \rightarrow \kk \rightarrow 0.\end{equation}  
in which $\ker(\partial_{i-1}) = \Omega_{\chi}^i(\kk)$. We can construct this by taking for each $i$ a short exact sequence
\[0 \rightarrow \Omega^{i+1}_{\chi}(\kk) \stackrel{\rho_i}{\rightarrow} Q_{i} \stackrel{\pi_i}{\rightarrow} \Omega^{i}_{\chi}(\kk) \rightarrow 0\] and setting $\partial_i:= \rho_{i} \pi_{i+1}$. For each $i$ let 

$$\delta_i: \Hom_{\kk G}(Q_i,\kk) \rightarrow \Hom_{\kk G}(Q_{i+1}, \kk)$$
denote the map induced by $\partial_i$. 

Our main tool will be as following:

\begin{prop}\label{cohomtool} Let $N$ be a $\kk G$-module. Then we have
\begin{enumerate}
\item[(i)] $H^0_{\chi}(G,N) = N^G$;
\item[(ii)] $H^i_{\chi}(G,N) \cong \underline{\Hom}_{\kk G}^{\chi}(\Omega^i_{\chi}(\kk), N)$.
\end{enumerate}
\end{prop}

The proof is the same as in the case $\chi = \{1\}$, but we give a sketch for lack of a good reference to this proof.

\begin{proof}
We first show that for each $i \geq 0$, $$\ker(\delta_i)\cong \Hom_{\kk G}(\Omega^i_{\chi}(\kk),N).$$
To see this, let $\phi \in \ker(\delta_i) \subseteq \Hom_{\kk G}(Q_i,N)$. For $x \in \Omega^i_{\chi}(\kk)$, choose $q \in Q_i$ such that $\pi_i(q) = x$ and define $\hat{\phi} (x) = \phi(q)$. Then $\hat{\phi} \in \Hom_{\kk G}(\Omega^i_{\chi}(\kk, N))$. The assignment $\phi \rightarrow \hat{\phi}$ is well-defined: for if $q' \in Q_i$ with $\pi_i(q') = x$ and $\tilde{\phi}(x):= \phi(q')$, then since $q-q' \in \ker(\pi_i)$ we get $q-q' \in \im(\partial_i)$ and $\phi(q-q') = 0$ since $\phi \in \ker(\delta_i)$. Conversely, given $\phi \in \Hom_{\kk G}(\Omega^i_{\chi}(\kk), N)$ we can define $\hat{\phi} = \phi \circ \pi_i \in \ker(\delta_i)$. It's easy to see that the two assignments are inverse to each other.

This in particular shows that (i) holds, since $\Hom_{\kk G}(\kk, N) \cong N^G$. We now show that $\im(\delta_{i-1})$ consists of the homomorphisms in $\Hom_{\kk G}(\Omega^i(\kk), N)$ which factor through a module which is projective relative to $\chi$. To see this, first suppose $\phi \in \im(\delta_{i-1}) \subseteq \Hom_{\kk G}(Q_i, N)$, say $\phi = \psi \circ \partial_{i-1}$ where $\psi \in \Hom_{\kk G}(Q_{i-1}, N)$.
Then with $x \in \Omega^i_{\chi}(\kk)$ and $q, \hat{\phi}$ as before we note that
\[\psi \circ \rho_{i-1}(x) = \psi \circ \rho _{i-1}\circ \pi_i (q) = \psi \circ \partial_i (q) = \phi(q) = \hat{\phi}(x) \] which shows that $\hat{\phi}$ factors through the module $Q_{i-1}$ which is projective relative to $\chi$. Conversely, if $\phi \in \Hom_{\kk G}(\Omega^i_{\chi}, \kk)$ factors through any module which is projective relative to $\chi$, then it factors through $Q_{i-1}$, because $\rho_{i-1}$ is injective and $Q_{i-1}$ is also an injective module with respect to $\chi$ by Lemma \ref{omnibus}.

\end{proof}

One can define a pairing $\smile: H_{\chi}^i(G,\kk) \otimes H_{\chi}^j(G,\kk) \rightarrow H_{\chi}^{i+j}(G,\kk)$ in a few different ways. On the one hand, elements of $H_{\chi}^*(G,\kk) = \Ext^*_{\kk G, \chi}(\kk,\kk)$ can be viewed as equivalence classes of extensions of $\kk$ by $\kk$ split over $\chi$, and the usual Yoneda splice gives the required pairing; see \cite[Section~2.6,3.9]{Benson1} for details in the case $\chi$ consisting of only one subgroup. Some other constructions in the case $\chi = \{1\}$ are given in \cite{CarlsonCohom}, and all of these extend in a natural way to arbitrary $\chi$. Happily, all these methods give the same construction. In the present article we will use the following construction: recall that

\[H_{\chi}^i(G,\kk) \cong \underline{\Hom}_{\kk G}^{\chi}(\Omega^i_{\chi}(\kk), \kk).\]
Similarly
\[H_{\chi}^j(G,\kk) = \underline{\Hom}_{\kk G}^{\chi}(\Omega^j_{\chi}(\kk), \kk) \cong \underline{\Hom}_{\kk G}^{\chi}(\Omega^{i+j}_{\chi}(\kk),  \Omega^i_{\chi}(\kk)) \]
with the second isomorphism arising from Proposition \ref{suspendhom}. Therefore we may define a product as follows: for $\alpha \in H_{\chi}^i(G,\kk)$ and $\beta \in H_{\chi}^j(G,\kk)$ choose $f \in \Hom_{\kk G}(\Omega^i_{\chi}(\kk),\kk)$, $g \in \Hom_{\kk G}(\Omega^j_{\chi}(\kk),\kk)$ respresenting $\alpha, \beta$ respectively. Then $\Omega^i_{\chi}(g) \in \Hom_{\kk G}(\Omega_{\chi}^{i+j}(\kk),\Omega^i_{\chi}(\kk))$, so that 
\[f \circ \Omega^i_{\chi}(g) \in \Hom_{\kk G}(\Omega^{i+j}_{\chi}(\kk), \kk).\]

We take $\alpha \smile \beta$ to be the cohomology class represented by $f \circ \Omega^i_{\chi}(g)$. This is called the \emph{cup product} of $\alpha$ and $\beta$.

\section{Representations of $C_2 \times C_2$}

In this section, let $G = \langle \sigma, \tau \rangle$ denote the Klein four-group, and let $\kk$ be a field of characteristic 2 (not necessarily algebraically closed). We set $\chi = \{H_1,H_2,H_3\}$, the set of all proper nontrivial subgroups of $G$, where $H_1 = \langle \sigma \rangle, H_2 = \langle \tau \rangle, H_3 = \langle \sigma \tau \rangle$. 

Let $X:= \sigma - 1 \in \kk G$, $Y:= \tau - 1 \in \kk G$. Then $X^2=Y^2=0$, $\kk G$ is isomorphic to the quotient ring
\[R:= \kk[X,Y]/(X^2,Y^2),\] and $\kk G$-modules can be viewed as $R$ modules.
We will describe $R$-modules by means of the diagrams for modules first introduced by Alperin in \cite{AlperinDiagrams}. In these diagrams, nodes represent basis elements, and two nodes labelled $a$ and $b$ are joined by a south-west directed arrow if $Xa=b$, and by a south-east directed arrow if $Ya=b$. If no south-west arrow begins at $a$ then it is understood that $Xa=0$, similarly for $Y$.

Our statement of the classification of $\kk G$-modules resembles that found in \cite{Kleinfoursurvey} and we recommend this reference as an easily accessible proof.

\begin{prop}\label{classify} Let $M$ be an indecomposable $\kk G$-module. Then $M$ is isomorphic to one of the following:

\begin{enumerate}
\item The module $V_{2n+1}$ $(n \geq 0)$, with odd dimension $2n+1$ and diagram

\begin{tikzpicture}
\draw[fill] (0,0) circle[radius = 0.05];
\node[above] at (0,0) {$a_0$};
\draw[fill] (2,0) circle[radius = 0.05];
\node[above] at (2,0) {$a_1$};
\draw[fill] (1,-1) circle[radius = 0.05];
\node[below] at (1,-1) {$b_1$};
\node at (3,-0.75){$\ldots$};

\draw[fill] (4,0) circle[radius = 0.05];
\node[above] at (4,0) {$a_{n-1}$};
\draw[fill] (6,0) circle[radius = 0.05];
\node[above] at (6,0) {$a_n$};
\draw[fill] (5,-1) circle[radius = 0.05];
\node[below] at (5,-1) {$b_n$};

\draw[->] (0,0) -- (0.5,-0.5);
\draw (0.5,-0.5) -- (1,-1);
\draw[->] (2,0) -- (1.5,-0.5);
\draw (1.5,-0.5) -- (1,-1);
\draw[->] (2,0) -- (2.5,-0.5);

\draw[->] (4,0) -- (4.5,-0.5);
\draw (4.5,-0.5) -- (5,-1);
\draw[->] (6,0) -- (5.5,-0.5);
\draw (5.5,-0.5) -- (5,-1);
\draw[->] (4,0) -- (3.5,-0.5);
\end{tikzpicture}

\item The module $V_{-(2n+1)}$ ($n \geq 0$), with odd dimension $2n+1$ and diagram

\begin{tikzpicture}

\draw[fill] (0,-1) circle[radius = 0.05];
\node[below] at (0,-1) {$b_0$};
\draw[fill] (2,-1) circle[radius = 0.05];
\node[below] at (2,-1) {$b_1$};
\draw[fill] (1,0) circle[radius = 0.05];
\node[above] at (1,0) {$a_1$};
\node at (3,-0.75){$\ldots$};

\draw[fill] (4,-1) circle[radius = 0.05];
\node[below] at (4,-1) {$b_{n-1}$};
\draw[fill] (6,-1) circle[radius = 0.05];
\node[below] at (6,-1) {$b_n$};
\draw[fill] (5,0) circle[radius = 0.05];
\node[above] at (5,0) {$a_n$};

\draw[->] (1,0) -- (1.5,-0.5);
\draw (1.5,-0.5) -- (2,-1);
\draw[->] (1,0) -- (0.5,-0.5);
\draw (0.5,-0.5) -- (0,-1);
\draw (2,-1) -- (2.5,-0.5);

\draw[->] (5,0) -- (5.5,-0.5);
\draw (5.5,-0.5) -- (6,-1);
\draw[->] (5,0) -- (4.5,-0.5);
\draw (4.5,-0.5) -- (4,-1);
\draw (3.5,-0.5) -- (4,-1);

\end{tikzpicture}

Note that $V_1 \cong V_{-1} \cong \kk$, with trivial $G$-action, but otherwise these modules are pairwise non-isomorphic.

\item The module $V_{2n,\infty}$, $(n \geq 1)$, with even dimension $2n$ and diagram

\begin{tikzpicture}

\draw[fill] (0,-1) circle[radius = 0.05];
\node[below] at (0,-1) {$b_1$};
\draw[fill] (2,-1) circle[radius = 0.05];
\node[below] at (2,-1) {$b_2$};
\draw[fill] (1,0) circle[radius = 0.05];
\node[above] at (1,0) {$a_1$};
\node at (3,-0.5){$\ldots$};

\draw[fill] (4,0) circle[radius = 0.05];
\node[above] at (4,0) {$a_{n-1}$};
\draw[fill] (6,0) circle[radius = 0.05];
\node[above] at (6,0) {$a_n$};
\draw[fill] (5,-1) circle[radius = 0.05];
\node[below] at (5,-1) {$b_n$};

\draw[->] (1,0) -- (1.5,-0.5);
\draw (1.5,-0.5) -- (2,-1);
\draw[->] (1,0) -- (0.5,-0.5);
\draw (0.5,-0.5) -- (0,-1);
\draw (2,-1) -- (2.5,-0.5);

\draw[->] (4,0) -- (4.5,-0.5);
\draw (4.5,-0.5) -- (5,-1);
\draw[->] (6,0) -- (5.5,-0.5);
\draw (5.5,-0.5) -- (5,-1);
\draw (4,0) -- (3.5,-0.5);

\end{tikzpicture}

\item The module $V_{2n,\theta}$, $(n \geq 1)$, with even dimension $2n$ and diagram,

\begin{tikzpicture}
\draw[fill] (0,0) circle[radius = 0.05];
\node[above] at (0,0) {$a_1$};
\draw[fill] (2,0) circle[radius = 0.05];
\node[above] at (2,0) {$a_2$};
\draw[fill] (1,-1) circle[radius = 0.05];
\node[below] at (1,-1) {$b_1$};
\node at (3,-0.5){$\ldots$};

\draw[->] (0,0) -- (0.5,-0.5);
\draw (0.5,-0.5) -- (1,-1);
\draw[->] (2,0) -- (1.5,-0.5);
\draw (1.5,-0.5) -- (1,-1);
\draw (2,0) -- (2.5,-0.5);

\draw[dashed,->] (0,0) -- (0,-1);
\node[left] at (0,-0.5) {$\theta$};

\draw[fill] (4,-1) circle[radius = 0.05];
\node[below] at (4,-1) {$b_{n-1}$};
\draw[fill] (6,-1) circle[radius = 0.05];
\node[below] at (6,-1) {$b_n$};
\draw[fill] (5,0) circle[radius = 0.05];
\node[above] at (5,0) {$a_n$};

\draw[->] (5,0) -- (5.5,-0.5);
\draw (5.5,-0.5) -- (6,-1);
\draw[->] (5,0) -- (4.5,-0.5);
\draw (4.5,-0.5) -- (4,-1);
\draw (4,-1) -- (3.5,-0.5);
\end{tikzpicture}

Here, $\theta(x) = \sum_{i=0}^n \lambda_i x^{n-i}$ is a power of an irreducible monic polynomial with coefficients in $\kk$ and the dotted line labelled by $\theta$ indicates that $Xa_1 = \sum_{i=1}^n {\lambda_i} b_i$.

\item The projective indecomposable module $P$, with dimension 4 and diagram

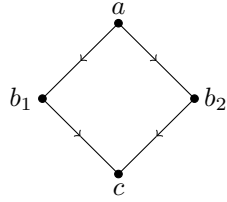
\begin{tikzpicture}
\draw[fill] (0,-1) circle[radius = 0.05];
\node[left] at (0,-1) {$b_1$};
\draw[fill] (2,-1) circle[radius = 0.05];
\node[right] at (2,-1) {$b_2$};
\draw[fill] (1,0) circle[radius = 0.05];
\node[above] at (1,0) {$a$};
\draw[fill] (1,-2) circle[radius = 0.05];
\node[below] at (1,-2) {$c$};

\draw[->] (1,0) -- (1.5,-0.5);
\draw (1.5,-0.5) -- (2,-1);
\draw[->] (1,0) -- (0.5,-0.5);
\draw (0.5,-0.5) -- (0,-1);
\draw[->] (2,-1) -- (1.5,-1.5);
\draw (1.5,-1.5) -- (1,-2);
\draw[->] (0,-1) -- (0.5,-1.5);
\draw (0.5,-1.5) -- (1,-2);

\end{tikzpicture}

\end{enumerate}
 
\end{prop}
 
The following, also taken from \cite{Kleinfoursurvey}, may be proved directly from the classification above.

\begin{prop}\label{dual} Let $M$ be an indecomposable $\kk G$-module. Then we have

\begin{enumerate} 
\item $M \cong M^*$ if $M$ is even-dimensional.
\item $M^* \cong V_{-(2n+1)}$ if $M \cong V_{2n+1}$ is odd dimensional.
\item $M^* \cong V_{2n+1}$ if $M \cong V_{-(2n+1)}$ is odd-dimensional.
\end{enumerate}

\end{prop}
Clearly (3) follows from (2) above, but we include it for completeness. In addition,

\begin{prop}\label{hellerstandard} Let $M$ be an indecomposable $\kk G$-module. Then we have

\begin{enumerate} 
\item $\Omega(M) \cong M$ if $M$ is even-dimensional.
\item $\Omega^{-1}(M) \cong V_{-(2n+3)}$ if $M \cong V_{-(2n+1)}$ is odd dimensional.
\item $\Omega(M) \cong V_{2n+3}$ if $M \cong V_{2n+1}$ is odd-dimensional.
\end{enumerate}
\end{prop}

Again (3) follows from (2) when we take into account that $\Omega(M)^* \cong \Omega^{-1}(M^*)$ in general.

\subsection{Relative shifts}
The goal of this subsection is to prove Theorem \ref{heller}.

Among the indecomposable $\kk G$-modules listed in the previous section, only four are projective relative to $\chi$. These are the projective indecomposable $P$, and the three modules $V_{2,\infty}$, $V_{2,x}$ and $V_{2,x+1}$. Here the last two are the indecomposable modules $V_{2,\theta}$ where $\theta(x)$ is the monic irreducible $x$ or $x+1 \in \kk[x]$. Note that $\tau$ acts trivially on $V_{2,\infty} = \kk_{H_2} \uparrow^G$, while $\sigma$ acts trivially on $V_{2,x} = \kk_{H_1} \uparrow^G$ and $\sigma \tau$ acts trivially on $V_{2,x+1} = \kk_{H_3} \uparrow^G$. As these three play in important role in what follows, we denote them by $Q_{\tau}, Q_{\sigma}$ and $Q_{\sigma \tau}$ respectively. We set $Q = Q_{\sigma} \oplus Q_{\tau} \oplus Q_{\sigma \tau}$. 

We begin by considering odd-dimensional modules. 
\begin{Lemma}\label{negodd} Let $n \geq 0$:
\begin{enumerate} 
\item The relative projective cover of $V_{-(2n+1)}$ is $Q \oplus nP$.
\item We have $\Omega_{\chi}(V_{-(2n+1)}) \cong V_{-(2n+5)}$.
\end{enumerate} 
\end{Lemma}

\begin{proof} Let $M \cong V_{-(2n+1)}$ and let $\pi: N \rightarrow M$ be its relative projective cover with respect to $\chi$. $N$ must decompose as a direct sum of modules of the form $P$, $Q_{\sigma}$, $Q_{\tau}$ and $Q_{\sigma \tau}$.  

Let $a_1,a_2, \ldots, a_n, b_0, b_1, \ldots, b_n$ be a basis of $M$, with action given by the diagram as in Proposition \ref{classify}. Since $\pi$ is a surjective $\kk G$-map and no $a_i$ is fixed by any element of $G$, the same must be true of their unique pre-images. The modules $Q_{\sigma}$, $Q_{\tau}$ and $Q_{\sigma \tau}$ all have non-trivial kernels. Therefore $N$ contains at least $n$ copies of $P$.

On the other hand, we have, for any $i$,
\begin{equation}\label{Mres} M \downarrow_{H_i} \cong \kk_{H_i} \oplus n \kk H_i \end{equation}
The restrictions to $H_1$ of $P$, $Q_{\tau}$ and $Q_{\sigma \tau}$ contain no trivial $H_1$-summands. So $N$ must contain a direct summand isomorphic to $Q_{\sigma}$  if $\pi$ is to split on restriction to $H_1$. A similar argument (restricting to $H_2, H_3$) shows that $N$ must contain summands isomorphic to $Q_{\tau}$ and $Q_{\sigma \tau}$.

We will construct a surjective $\kk G$-homomorphism $Q \oplus nP \rightarrow M$. The following diagrams label the basis elements:

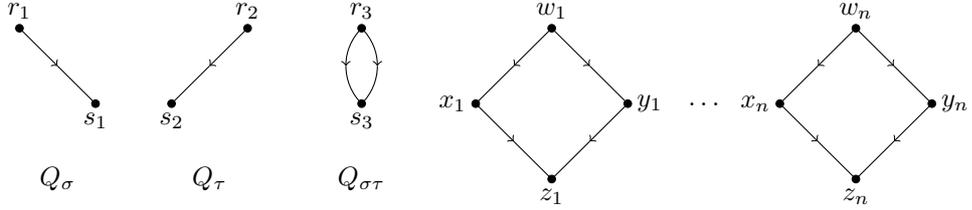
\begin{tikzpicture}
\draw[fill] (0,0) circle[radius=0.05];
\draw[fill] (1,-1) circle[radius=0.05];
\draw[->] (0,0) -- (0.5,-0.5);
\draw (0.5,-0.5) -- (1,-1);
\node at (0.5,-2) {$Q_{\sigma}$};
\node[above] at (0,0) {$r_1$};
\node[below] at (1,-1) {$s_1$};

\draw[fill] (3,0) circle[radius=0.05];
\draw[fill] (2,-1) circle[radius=0.05];
\draw[->] (3,0) -- (2.5,-0.5);
\draw (2.5,-0.5) -- (2,-1);
\node at (2.5,-2) {$Q_{\tau}$};
\node[above] at (3,0) {$r_2$};
\node[below] at (2,-1) {$s_2$};

\draw[fill] (4.5,0) circle[radius=0.05];
\node[above] at (4.5,0) {$r_3$};
\draw[fill]  (4.5,-1) circle[radius=0.05];
\node[below] at (4.5,-1) {$s_3$};
\draw[->] (4.5,0) arc (135:180:0.707);
\draw (4.293,-0.5) arc (180:225:0.707);
\draw (4.5,-1) arc (-45:0:0.707);
\draw[<-] (4.707,-0.5) arc (0:45:0.707);
\node at (4.5,-2) {$Q_{\sigma \tau}$};

\draw[fill] (6,-1) circle[radius = 0.05];
\node[left] at (6,-1) {$x_1$};
\draw[fill] (8,-1) circle[radius = 0.05];
\node[right] at (8,-1) {$y_1$};
\draw[fill] (7,0) circle[radius = 0.05];
\node[above] at (7,0) {$w_1$};
\draw[fill] (7,-2) circle[radius = 0.05];
\node[below] at (7,-2) {$z_1$};

\draw[->] (7,0) -- (7.5,-0.5);
\draw (7.5,-0.5) -- (8,-1);
\draw[->] (7,0) -- (6.5,-0.5);
\draw (6.5,-0.5) -- (6,-1);
\draw[->] (8,-1) -- (7.5,-1.5);
\draw (7.5,-1.5) -- (7,-2);
\draw[->] (6,-1) -- (6.5,-1.5);
\draw (6.5,-1.5) -- (7,-2);

\node at (9,-1) {$\ldots$};

\draw[fill] (10,-1) circle[radius = 0.05];
\node[left] at (10,-1) {$x_n$};
\draw[fill] (12,-1) circle[radius = 0.05];
\node[right] at (12,-1) {$y_n$};
\draw[fill] (11,0) circle[radius = 0.05];
\node[above] at (11,0) {$w_n$};
\draw[fill] (11,-2) circle[radius = 0.05];
\node[below] at (11,-2) {$z_n$};

\draw[->] (11,0) -- (11.5,-0.5);
\draw (11.5,-0.5) -- (12,-1);
\draw[->] (11,0) -- (10.5,-0.5);
\draw (10.5,-0.5) -- (10,-1);
\draw[->] (12,-1) -- (11.5,-1.5);
\draw (10.5,-1.5) -- (11,-2);
\draw[->] (10,-1) -- (10.5,-1.5);
\draw (11.5,-1.5) -- (11,-2);

\end{tikzpicture}
(the diagram for $Q_{\sigma \tau}$ is not as described in Proposition \ref{classify}, but self-explanatory.) We now define a linear map $\pi: Q \oplus nP \rightarrow M$ by
\begin{itemize}
\item $\pi(w_i) = a_i$ for $i=1, \ldots, n$.
\item $\pi(x_i) = b_{i-1}$ for   $i=1, \ldots, n$.
\item $\pi(y_i) = b_i$  for $i=1, \ldots, n$.
\item $\pi(z_i) = 0$ for  $i=1, \ldots, n$.
\item $\pi(s_i) = 0$ for $i=1,2,3$.
\item $\pi(r_1) = \pi(r_3) = a_0$.
\item $\pi(r_2) = a_n$.
\end{itemize}
 
The reader should check that $\pi$ is a $\kk G$-homomorphism. The kernel of $\pi$ is spanned by 
$$\{z_i: i = 1, \ldots, n\} \cup \{ s_1,s_2,s_3\} \cup \{x_i+y_{i-1}: i=2, \ldots, n\} \cup \{x_1+r_1,x_1+r_3,y_n+r_2\}.$$
It has dimension $2n+5$, and the fixed-point space within this module is spanned by $\{z_1,z_2, \ldots, z_n,s_1,s_2,s_3\}$, so it has dimension $n+3$. It is easily checked that no element of the kernel outside of the fixed-point space is fixed by any subgroup $H_i$. Therefore $$\ker(\pi) \downarrow_{H_i} \cong \kk_{H_i} \oplus (n+2)\kk H_i$$ for any $i$. 
This, combined with (\ref{Mres}) and the fact that
$$(Q \oplus nP) \downarrow_{H_i} \cong 2 \kk_{H_i} \oplus (2n+2) \kk H_i$$ shows that $\pi$ splits on restriction to any $H_i$. The construction ensures the minimality of $Q \oplus nP$, so $Q \oplus nP = N$, proving (1). Further, $\Omega_{\chi}(M) = \ker(\pi)$, and the classification of $\kk G$-modules, together with the fact that $\ker(\pi)$ must be indecomposable, implies that $\ker(\pi) \cong V_{-(2n+5)}$, proving (2).  
\end{proof}

The following follows immediately the above using Propositions \ref{dual} and \ref{omegaomnibus}(3).
\begin{Lemma} Let $n \geq 0$:
Then we have $\Omega_{\chi}(V_{(2n+5)}) \cong V_{(2n+1)}$.
\end{Lemma}

To complete the picture for odd-dimensional modules, it remains only to show that 
\begin{Lemma} Let $M \cong V_3$. Then:
\begin{enumerate}
\item The relative projective cover of $M$ is $Q$;
\item We have $\Omega_{\chi}(M) \cong V_{-3}$.
\end{enumerate}
\end{Lemma}

\begin{proof} We have $M \downarrow_{H_i} \cong \kk_{H_i} \oplus \kk H_i$, for $i=1,2,3$, so once more the projective cover must contain a summand isomorphic to $Q$. We shall construct a $\kk G$-homomorphism $\pi: Q \rightarrow M$.  We retain the notation for a basis of $Q$ used in Lemma \ref{negodd}; a basis for $M$ is $\{a_0,a_1,b_1\}$ with action given as in the classification.

Define:
\begin{itemize}
\item $\pi(r_1) =a_0$ 
\item $\pi(r_2) = a_1$
\item $\pi(r_3) = a_0+a_1$.
\item $\pi(s_1)=\pi(s_2)=\pi(s_3)=b_1.$
\end{itemize}
The reader should check this is a $\kk G$-homomorphism. The kernel of $\pi$ is spanned by $\{s_1+s_2,s_2+s_3,r_1+r_2+r_3\}$, and the fixed-point space of the kernel is two-dimensional, spanned by $\{s_1+s_3, s_2+s_3\}$. Noting that 
$$X(r_1+r_2+r_3) = s_2+s_3, Y(r_1+r_2+r_3) = s_1+s_3,$$ we see that the kernel of $\pi$ is indecomposable, and as a $\kk G$-module is isomorphic to $V_{-3}$. Therefore
\[\ker(\pi)_{H_i} \oplus \kk_{H_i} \oplus \kk H_i\] for all $i$, from which we deduce that $\pi$ splits on restriction to each $H_i$. Our construction ensures the minimality of $Q$, so $Q$ is indeed the relative projective cover of $M$, proving (1), and $\ker(\pi) = \Omega_{\chi}(M) \cong V_{-3}$, proving (2).
\end{proof}

We now turn to even dimensional modules. Note that $V_{2, \infty} = Q_{\tau}$ is already projective relative to $\chi$, so $\Omega_{\chi}(V_{2,\infty})$ is not defined.

\begin{Lemma}\label{evenheller} Let $n \geq 2$ and $M \cong V_{2n, \infty}$. Then:
\begin{enumerate}
\item The relative projective cover of $M$ is $2Q_{\tau} \oplus (n-1) P$;
\item We have $\Omega_{\chi} (M) \cong M$.
\end{enumerate}
\end{Lemma}

\begin{proof} Let $\pi: N \rightarrow M$ be the relative projective cover of $M$. Notice that 
\begin{equation}\label{Mres13} M \downarrow_{H_i} = n \kk H_i\end{equation} for $i=1,3$ whereas 
\begin{equation}\label{Mres2} M \downarrow_{H_2} = 2 \kk_{H_2} \oplus (n-1) \kk H_2.\end{equation}
So if $\pi: N \rightarrow M$ is to split on restriction to $H_2$, $N$ must contain a pair of direct summands isomorphic to $Q_{\tau}$.
On the other hand, retaining the notation from Proposition \ref{classify}, the basis elements $a_1, \ldots, a_{n-1}$ are not fixed by any element of $G$, so the same must be true of their unique pre-images in $N$. From this it follows that $N$ must contain $n-1$ direct summands isomorphic to $P$.

We will construct a $\kk G$-homomorphism $2 Q_{\tau} \oplus (n-1) P \rightarrow M$. The following diagram gives the labelling for a basis of the domain:

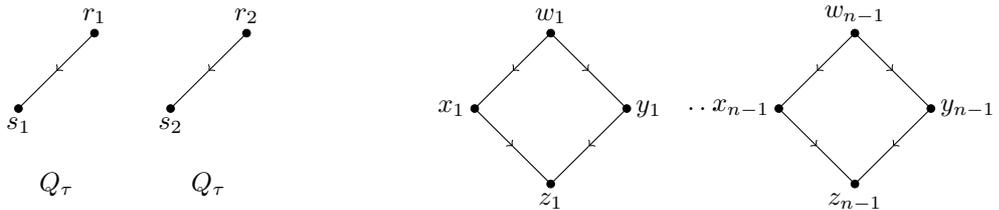
\begin{tikzpicture}
\draw[fill] (1,0) circle[radius=0.05];
\draw[fill] (0,-1) circle[radius=0.05];
\draw[->] (1,0) -- (0.5,-0.5);
\draw (0.5,-0.5) -- (0,-1);
\node at (0.5,-2) {$Q_{\tau}$};
\node[above] at (1,0) {$r_1$};
\node[below] at (0,-1) {$s_1$};

\draw[fill] (3,0) circle[radius=0.05];
\draw[fill] (2,-1) circle[radius=0.05];
\draw[->] (3,0) -- (2.5,-0.5);
\draw (2.5,-0.5) -- (2,-1);
\node at (2.5,-2) {$Q_{\tau}$};
\node[above] at (3,0) {$r_2$};
\node[below] at (2,-1) {$s_2$};

\draw[fill] (6,-1) circle[radius = 0.05];
\node[left] at (6,-1) {$x_1$};
\draw[fill] (8,-1) circle[radius = 0.05];
\node[right] at (8,-1) {$y_1$};
\draw[fill] (7,0) circle[radius = 0.05];
\node[above] at (7,0) {$w_1$};
\draw[fill] (7,-2) circle[radius = 0.05];
\node[below] at (7,-2) {$z_1$};

\draw[->] (7,0) -- (7.5,-0.5);
\draw (7.5,-0.5) -- (8,-1);
\draw[->] (7,0) -- (6.5,-0.5);
\draw (6.5,-0.5) -- (6,-1);
\draw[->] (8,-1) -- (7.5,-1.5);
\draw (7.5,-1.5) -- (7,-2);
\draw[->] (6,-1) -- (6.5,-1.5);
\draw (6.5,-1.5) -- (7,-2);

\node at (9,-1) {$\ldots$};

\draw[fill] (10,-1) circle[radius = 0.05];
\node[left] at (10,-1) {$x_{n-1}$};
\draw[fill] (12,-1) circle[radius = 0.05];
\node[right] at (12,-1) {$y_{n-1}$};
\draw[fill] (11,0) circle[radius = 0.05];
\node[above] at (11,0) {$w_{n-1}$};
\draw[fill] (11,-2) circle[radius = 0.05];
\node[below] at (11,-2) {$z_{n-1}$};

\draw[->] (11,0) -- (11.5,-0.5);
\draw (11.5,-0.5) -- (12,-1);
\draw[->] (11,0) -- (10.5,-0.5);
\draw (10.5,-0.5) -- (10,-1);
\draw[->] (12,-1) -- (11.5,-1.5);
\draw (10.5,-1.5) -- (11,-2);
\draw[->] (10,-1) -- (10.5,-1.5);
\draw (11.5,-1.5) -- (11,-2);

\end{tikzpicture}
 
We define:
\begin{itemize}
\item $\pi(w_i) = a_i$ for $i=1, \ldots, n-1$.
\item $\pi(x_i) = b_i$ for $i=1, \ldots, n-1$.
\item $\pi(y_i) = b_{i+1}$ for $i=1, \ldots, n-1$.
\item $\pi(z_i)=0$ for $i=1, \ldots, n-1$.
\item $\pi(r_1) = b_1$.
\item $\pi(s_1) = 0$.
\item $\pi(r_2) = a_n$.
\item $\pi(s_2) = b_n$.
\end{itemize}

The reader should check that $\pi$ is a $\kk G$-homomorphism. The kernel of $\pi$ is spanned by
\[\{z_i: i=1, \ldots, n-1\} \cup \{x_i+y_{i-1}: i=2, \ldots, n-1\} \cup \{s_1,x_1+r_2, y_{n-1}+s_2\}.\] 
This has dimension $2n$. The fixed points within this module are spanned by
\[\{z_i: i=1, \ldots, n-1\} \cup \{s_1\}. \]
These span the fixed points of $H_1$ and $H_3$, while $H_2$ has a fixed point space of dimension $n+1$, spanned by the above and $y_{n+1}+s_2$. Therefore we have
\[\ker(\pi) \downarrow_{H_i} \cong n \kk H_i\] for $i=1,3$ and
\[\ker(\pi) \downarrow_{H_2} \cong 2\kk_{H_2}\oplus (n-1) \kk H_2.\]
Note that
\[(2Q_{\tau} \oplus (n-1)P) \downarrow_{H_i} \cong 2n \kk H_i \] for $i=1,3$ and
\[(2Q_{\tau} \oplus (n-1)P) \downarrow_{H_2} \cong 4 \kk_{H_2}  \oplus (2n-2) \kk H_i. \]
Thus, $\pi$ splits on restriction to each $H_i$. The construction ensures the minimality of $2 Q_{\tau} \oplus (n-1)P$, so this is equal to $N$ and we have (1). Further, $\ker(\pi) = \Omega_{\chi}(M)$ must be indecomposable. By the classification (looking at the dimension of the fixed point space of each subgroup of $G$ to distinguish among modules of even dimension) we must have $\Omega_{\chi}(M) \cong M$ as required for (2). 
\end{proof}

Notice that if $\theta(x) = x^n$, then $V_{2n,\theta}$ can be obtained from $V_{2n,\infty}$ by applying the automorphism of $G$ which swaps $\sigma$ and $\tau$. Similarly if $\theta(x)=(x+1)^n$, then $V_{2n,\theta}$ can be obtained from $V_{2n,\infty}$ by applying the automorphism of $G$ which swaps $\sigma \tau$ and $\tau$. We therefore obtain immediately from Lemma \ref{evenheller} above that $\Omega_{\chi}(M) = M$ if $M$ is one of these. 

It remains only to prove the following:
\begin{Lemma} Let $n \geq 1$ and let $M \cong V_{2n,\theta}$, where $\theta$ is neither $x^n$ nor $(x+1)^n$. Then:
\begin{enumerate} 
\item The relative projective cover of $M$ is $nP$;
\item $\Omega_{\chi}(M) \cong M$.
\end{enumerate}
\end{Lemma}

\begin{proof} Observe that $M \downarrow_{H_i} = n \kk H_i$ for each $i$. The proof of \cite[Proposition~3.1]{Kleinfoursurvey} shows that the projective (as opposed to relatively projective) cover of $M$ is $nP$ and $\Omega(M) \cong M$, so there is a surjective $\kk G$-homomorphism $\pi: nP \rightarrow M$ with kernel isomorphic to $M$. Noting that $nP \downarrow_{H_i} \cong 2n \kk H_i$ for each $i$, we see that $\pi$ splits on restriction to each $H_i$. On the other hand, if $N$ is a $\kk G$-module having $Q_{\tau}$ (resp. $Q_{\sigma}, Q_{\sigma \tau}$) as a direct summand then $N \downarrow_{H_i}$ contains a pair of trivial $\kk H_i$-modules as direct summand, and no surjective homomorphism $N \rightarrow M$ may split. This shows the minimality of the dimension of $nP$ among relatively projective modules with a $\chi$-split epimorphism to $M$, i.e. we have proved (1). We also have
\[\Omega_{\chi}(M) = \ker(\pi) = \Omega(M) \cong M\] as required for (2).  
\end{proof}

\begin{rem} Combining all the Lemmas in this section with Proposition \ref{hellerstandard}, we obtain Theorem \ref{heller}.
\end{rem}

\subsection{Computing Cohomology}

In this subsection we will determine $H^i(G,N)$ for all $i \geq 0$ and for all indecomposable $\kk G$-modules $N$. First observe that if $N$ is projective relative to $\chi$, then $H^i(G,N)=0$ for all $i>0$: this is an immediate consequence of Proposition \ref{cohomtool}(ii). Further, recall from part (i) of the same that $H_{\chi}^0(G,N) = N^G$ for any $\kk G$-module. It follows that:

\begin{prop} Let $N \in \{P, Q_{\sigma}, Q_{\tau}, Q_{\sigma \tau}\}$. Then.
\[\dim(H_{\chi}^i(G,N)) = \left\{\begin{array}{cc} 1 & i=0, \\ 0 & \ \text{otherwise.} \end{array} \right.\]
\end{prop}

Now we consider even-dimensional modules which are not relatively projective. Recall that for $i>0$ we have
\[H_{\chi}^i(G,N) = \underline{\Hom}^{\chi}_{\kk G}(\Omega^i_{\chi}(\kk),N) \cong \underline{\Hom}^{\chi}_{\kk G}(\kk,\Omega_{\chi}^{-i}(N)) \cong \underline{\Hom}^{\chi}_{\kk G}(\kk,N) \cong N^G_{\chi}\]
using the fact that, for these modules $N$, we have $\Omega_{\chi}^{-i}(N) \cong N$.

We obtain by direct calculation:

\begin{prop} Let $N $ be an even-dimensional $\kk G$-module which is not projective relative to $\chi$. Then.
\[\dim(H_{\chi}^i(G,N)) = \left\{\begin{array}{cc} n & i=0 \\ n-1 & \ \text{otherwise} \end{array} \right.\]
if $N \cong V_{2n,\infty}$ or $N \cong V_{2n, \theta}$ where $\theta(x) = x^n$ or $\theta(x) = (x+1)^n$,
 while

\[\dim(H_{\chi}^i(G,N)) = n\] for any $i$, if $V \cong V_{2n,\theta}$ for some other choice of $\theta$.
\end{prop}

For odd-dimensional modules we proceed as follows. Let $N$ be an odd-dimensional indecomposble module and let $i>0$. Then
\[H_{\chi}^i(G,N) = \underline{\Hom}_{\kk G}(\Omega^i_{\chi}(\kk),N) \cong \underline{\Hom}_{\kk G}(\kk,\Omega_{\chi}^{-i}(N) \cong \underline{\Hom}_{\kk G}(\kk,\Omega^{2i}(N)) \cong \Omega^{2i}(N)^G_{\chi}\]
using Theorem \ref{heller}. Suppose $N \cong V_{2n+1}$ where $n \geq 0$. Then $\Omega^{2i}(N) \cong V_{2(n+2i)+1}$. A basis for $V_{2(n+2i)+1}$ is given by $\{a_0,a_1, \ldots, a_{n+2i},b_1, b_2, \ldots, b_{n+2i}\}$, with action given by the diagram in Proposition \ref{classify}. The $b_i$ are all fixed points, and in addition $a_0$ is fixed by $H_1$, $a_{n+2i}$ by $H_2$ and $a_0+a_1+ \ldots + a_{n+2i}$ by $H_3$. Therefore $b_1, b_{n+2i}$ and $b_1+b_2+ \ldots + b_{n+2i}$ lie in $\Omega^{2i}(N)^{G, \chi}$. We therefore have

\begin{prop} Let $N \cong V_{2n+1}$ for some $n \geq 0$. Then 
\begin{enumerate}
\item $\dim(H_{\chi}^0(G,N)) = n$ if $n > 0$, and 1 if $n=0$.

\item $\dim(H_{\chi}^i(G,N)) = \max(0,n+2i-3)$ for $i>0$.
\end{enumerate}

\end{prop}

\begin{rem} This includes \cite[Theorem~1.2]{PamukYalcin} as a special case ($n=0$).
\end{rem}
 For the remaining odd dimensional modules things are a little more complicated, since $\Omega^{2i}(N)$ eventually moves into the ``positive'' part of the spectrum. We begin by noting that if $n \geq 0$, then $V^{H_i}_{-(2n+1)} = V^G_{-(2n+1)}$  for all $i$. Therefore $(V_{-(2n+1)})^{G,\chi} = 0$. 

Now let $N \cong V_{-(2n+1)}$ where $n \geq 1$. For $i \leq n/2$ we have $\Omega^{2i}(N) \cong V_{-(2(n-2i)+1)}$. Therefore
\[H_{\chi}^i(G,N) = \underline{\Hom}_{\kk G}(\Omega^i_{\chi}(\kk),N) \cong \underline{\Hom}_{\kk G}(\kk,\Omega_{\chi}^{-i}(N) \cong \underline{\Hom}_{\kk G}(\kk,\Omega^{2i}(N)) \cong  \Omega^{2i}(N)^G.\]
For $i>n/2$ we have $\Omega^{2i}(N) \cong V_{2(2i-n)+1}$. We therefore obtain the following:
\begin{prop} Let $N \cong V_{-(2n+1)}$ where $n \geq 1$. Then
\[\dim(H_{\chi}^i(G,N)) = \left\{ \begin{array}{cc} n+1-2i & i \leq n/2 \\ 
\max(0,2i-n-3) & i > n/2. \end{array} \right.\]
\end{prop}

\subsection{Calculating cup products}
The aim of this section is to prove Theorem \ref{cupprod}. We begin with a lemma:

\begin{Lemma} Let $M \cong V_{-(2m+1)}$ and $N \cong V_{-(2n+1)}$ for some $m > n \geq 0$. Let $\phi \in \Hom_{\kk G}(M,N)$. Then 
\begin{enumerate}
\item $\im(\phi) \subseteq N^G$;
\item $M^G \subseteq \ker(\phi)$.
\end{enumerate}
\end{Lemma}

\begin{proof} Note first that $\phi(M^G) \subseteq N^G$ for arbitrary $G$ and $\kk G$-modules $M$ and $N$.
Let $a_1, a_2, \ldots, a_m, b_0, b_1, \ldots, b_m$ and $a'_1,a'_2, \ldots, a'_n, b'_0, b'_1, \ldots, b'_n$ be bases of $M$ and $N$ respectively, with action given by the diagrams in proposition \ref{classify}. Note that if $n=0$, then (1) is immediate. So suppose $n>0$ and (1) does not hold: then we can find a maximal $k \geq 1$ such that $\phi(a_k) \not \in N^G$.

We claim that $k=m$. To see this, write \[\phi(a_k) = \sum_{i=1}^n \lambda_i a'_i \ \mod \ N^G.\]
Then
\[\phi(b_k) =\phi(Ya_k) = Y\phi(a_k) = \sum_{i=1}^n \lambda_i b_i'.\]
If $k<m$ then also $$\phi(b_k) = \phi(Xa_{k+1} ) = X\phi(a_{k+1}) = 0$$ since $\phi(a_{k+1}) \in N^G$. So $\lambda_i = 0$ for all $i$ and $\phi(a_k) \in N^G$, a contradiction.

Now we claim that, for all $0 \leq j \leq n$, we have 
\begin{equation}\label{indhyp} 
\phi(a_{m-j}) = \sum_{i=j+1}^n \lambda_i a'_{i-j} \ \mod \ N^G 
\end{equation}
and $\lambda_i=0$ for $i=1, \ldots, j$. We prove this by induction on $j$. The base case $j=0$ is true by definition. Assuming the above for some $0 \leq j < n$ and noting that $n<m$, we have
\[\phi(b_{m-j-1}) = \phi(Xa_{m-j}) = X\phi(a_{m-j}) = \sum_{i=j+1}^n \lambda_i b'_{i-j-1}.\]
But \[\phi(b_{m-j-1}) = \phi(Ya_{m-j-1}) = Y\phi(a_{m-j-1}) \in YN = \langle b'_1, \ldots, b'_n \rangle\] which shows that $\lambda_{j+1}=0$. Therefore
\[\phi(b_{m-j-1}) = \sum_{i=j+2}^n \lambda_i b'_{i-j-1}\] which shows that
\[\phi(a_{m-j-1}) =  \sum_{i=j+2}^n \lambda_i a'_{i-j-1} \mod \ N^G\] proving our claim. Taking $j=n$ in (\ref{indhyp}) shows that $\phi(a_m) \in N^G$, a contradiction. This proves (1).\\

For (2), let $x \in M^G$. We may write
\[x = \sum_{i=0}^m \mu_i b_i\] for some coefficients $\mu_i$. Then
\[\phi(x) = \sum_{i=0}^m \mu_i \phi(b_i) =  \mu_0 \phi(X a_0) + \sum_{i=1}^m \mu_i \phi (Y a_{i-1}) = \mu_0 X \phi(a_0) + Y \phi(\sum_{i=1}^n \mu_i a_i) = 0\] by (1).
\end{proof}

The following is immediate:
\begin{Corollary}\label{trivcomp} Let $L \cong V_{-(2l+1)}$, $M \cong V_{-(2m+1)}$ and $N \cong V_{-(2n+1)}$ for some $l>m > n \geq0$. Let $\phi \in \Hom_{\kk G}(M,N)$ and $\psi \in \Hom_{\kk G}(L,M)$. Then $\phi \circ \psi = 0$. 
\end{Corollary}

We may now proceed with the proof of Theorem 2:
\begin{proof}
Let $i,j >0$. Let $\alpha \in H_{\chi}^i(G, \kk)$ and $\beta \in H^j_{\chi}(G,\kk)$. Choose $\phi \in \Hom_{\kk G}(\Omega_{\chi}^i(\kk), \kk)$ and $\psi \in \Hom_{\kk G}(\Omega_{\chi}^j(\kk), \kk)$, such that the equivalence classes 
$$[\phi] \in \underline{\Hom}^{\chi}_{\kk G}(\Omega_{\chi}^i(\kk), \kk), [\psi] \in \underline{\Hom}_{\kk G}^{\chi}(\Omega_{\chi}^j(\kk), \kk)$$
represent $\alpha$ and $\beta$ respectively. 
By definition, $\alpha \smile \beta$ is represented by $[\phi \circ \Omega_{\chi}^{i}(\psi)]$. By Lemma \ref{negodd} we have
\[\phi \in \Hom(V_{-(2i+1)},V_{(-1)}), \Omega^i_{\chi}(\psi) \in \Hom(V_{-(2i+2j+1)},V_{-(2i+1)})\] and by Corollary \ref{trivcomp} the composition of these two is the trivial map. 
\end{proof}

\bibliographystyle{plain}
\bibliography{MyBib}

\end{document}